\documentclass[11pt,a4paper,reqno]{amsart}

\usepackage{amsmath,amssymb,amsthm,amsfonts}
\usepackage{amsbsy}
\usepackage[utf8]{inputenc}
\usepackage{mathtools}
\usepackage{color}

\definecolor{LinkColor}{rgb}{0,0,0} 
\usepackage[colorlinks=true,linkcolor=LinkColor,citecolor=LinkColor,urlcolor= LinkColor, naturalnames, hyperindex, pdfstartview=FitH, bookmarksnumbered, plainpages]{hyperref} 
\usepackage{url}

\newtheorem{theorem}{Theorem}

\newtheorem{lemma}[theorem]{Lemma}
\newtheorem{proposition}[theorem]{Proposition}

\theoremstyle{definition}

\theoremstyle{remark}
\newtheorem{remark}[theorem]{Remark}
\newtheorem{example}[theorem]{Example}

\newcommand{\Sz}{\operatorname{Sz}}
\newcommand{\PSL}{\operatorname{PSL}}

\newcommand{\SL}{\operatorname{SL}}
\newcommand{\GL}{\operatorname{GL}}

\newcommand{\Z}{\mathbb{Z}}
\newcommand{\V}{\textup{V}}
\newcommand{\ZZ}{\mathbb{Z}}
\newcommand{\QQ}{\mathbb{Q}}
\newcommand{\C}{\mathbb{C}}
\newcommand{\lra}{\longrightarrow}
\newcommand{\sdp}{\rtimes}

\newcommand{\Q}{\mathbb Q}
\newcommand{\Ker}{\operatorname{Ker}}
\newcommand{\boxedcomp}[1]{\begin{tabular}{|c|} \hline \ensuremath{#1} \\ \hline \end{tabular}}
\newcommand{\boxxedcomp}[2]{\begin{tabular}{|c|} \hline \ensuremath{#1} \\ \hline \ensuremath{#2} \\ \hline \end{tabular}}

\begin{document}

\title[Algorithmic aspects of units in group rings]{Algorithmic aspects of units in group rings}
\author{Andreas B\"achle}
\address{Vakgroep Wiskunde, Vrije Universiteit Brussel, Pleinlaan 2, 1050 Brussels, Belgium}
\email{\href{mailto:abachle@vub.ac.be}{abachle@vub.ac.be}}
\author{Wolfgang Kimmerle}
\address{ Fachbereich Mathematik, IGT,  Universit\"{a}t Stuttgart, Pfaffenwaldring 57, 70550 Stuttgart, Germany}
\email{\href{mailto:kimmerle@mathematik.uni-stuttgart.de}{kimmerle@mathematik.uni-stuttgart.de}}
\author{Leo Margolis}
\address{Departamento de matem\'aticas, Facultad de matem\'aticas, Universidad de Murcia, 30100 Murcia, Spain}
\email{\href{mailto:leo.margolis@um.es}{leo.margolis@um.es}}
\thanks{The first author is a postdoctoral researcher of the FWO (Research Foundation Flanders). The third is supported by a Marie-Curie Individual Fellowship from EU project 705112-ZC}

\begin{abstract}
We describe the main questions connected to torsion subgroups in the unit group of integral group rings of finite groups and algorithmic methods to attack these questions. We then prove the Zassenhaus Conjecture for Amitsur groups and prove that any normalized torsion subgroup in the unit group of an integral group of a Frobenius complement is isomorphic to a subgroup of the group base. Moreover we study the orders of torsion units in integral group rings of finite almost quasisimple groups and the existence of torsion-free normal subgroups of finite index in the unit group.
\end{abstract}

\maketitle

\section{Introduction}
The study of the units of an integral group ring $\mathbb{Z}G$ for a finite group $G$ has begun in G. Higman's thesis \cite{HigmanThesis}. 
G.Higman classified the finite groups $G$ whose integral group ring has only trivial units. The aim of this article is to present recent work on the structure of torsion subgroups on the unit group of $\Z G$ which has been achieved especially with the aid of algorithmic tools. 
\vskip1em  

Denote by
\[\varepsilon: \mathbb{Z}G \rightarrow \mathbb{Z}, \ \sum_{g \in G} z_g g \mapsto \sum_{g \in G} z_g \]
the augmentation map. Being a ring homomorphism $\varepsilon$ maps units of $\mathbb{Z}G$ to units of $\mathbb{Z}$, so up to multiplication with $-1$ any unit in $\mathbb{Z}G$ has augmentation $1$ and it suffices to study the units of augmentation $1$ in $\mathbb{Z}G$. These so called normalized units units will be denoted by $\mathrm{V}(\mathbb{Z}G)$.

Though extensively studied, very few general theorems on the behaviour of finite subgroups in $\mathrm{V}(\mathbb{Z}G)$ are available. It is known that the order of a finite subgroup of $\mathrm{V}(\mathbb{Z}G)$ divides the order of $G$ \cite{ZK} and that the order of a torsion unit in $\mathrm{V}(\mathbb{Z}G)$ divides the exponent  of $G$ \cite{CohnLivingstone}. But it is not even known, if the orders of torsion units in $\mathrm{V}(\mathbb{Z}G)$ coincide with the orders of elements in $G$. For a long time the Isomorphism Problem, which asks if a ring isomorphism $\mathbb{Z}G \cong \mathbb{Z}H$ implies a group isomorphism $G \cong H$, was the focus of attention in the area, see e.g. \cite{RoggenkampScottIso}. A negative answer to this problem was finally given by M.Hertweck \cite{HertweckIso}. 

The main questions in the area of torsion units of integral group rings are given and inspired by three conjectures of H.~Zassenhaus. A subgroup $U \leq \mathrm{V}(\mathbb{Z}G)$ is called rationally conjugate to a subgroup of $G$, if there exist a subgroup $U'\leq G$ and a unit $x \in \mathbb{Q}G$ such that $x^{-1}Ux = U'$. Let $G$ be  finite group 
\
\vskip1em
\textbf{(ZC1)}  Units of finite order of $\mathrm{V}(\Z G)$ are rationally conjugate to 
elements of $G .$ 
\vskip1em

\textbf{(ZC2)}  Group bases, i.e. subgroups of $\mathrm{V}(\Z G)$ of the same order as $G$, are rationally conjugate.
\vskip1em

\textbf{(ZC3)}  A finite subgroup $H$ of $\mathrm{V}(\Z G)$ is rationally conjugate to a subgroup of $G .$ 
\vskip1em

Note that (ZC1) is still open. For (ZC2) and (ZC3) counterexamples are known (not only by the counterexample to the Isomorphism Problem). However  
they hold for important classes of finite groups, cf. Section 4.
\vskip1em 

In the last ten years computational tools have been developed to attack these 
questions. The fundamental tool is the HeLP - method which is now available as
a \texttt{GAP} - package \cite{HeLPPaper}. This package makes use of two external solvers for integral linear inequalities, namely 4ti2 \cite{4ti2} and normaliz \cite{Normaliz} which substantially improved the efficiency of \texttt{HeLP}. 
In Section 2 we describe the HeLP method as well as other algorithmic methods which have been developed in the last years to handle cases left open 
by \texttt{HeLP}.  
The remaining article is organized as follows. In Section 3 we prove (ZC1) for Amitsur groups. In the next section we survey recent results circulating around the Zassenhaus conjectures and isomorphism questions. We exhibit especially those results which have been established with the help of computational tools. We present as well the major research problems on torsion units of integral group rings.  
In Section 5 we show that finite subgroups of $\mathrm{V}(\Z G)$ are isomorphic to a subgroup of $G$ if $G$ is a Frobenius complement.     
In the last two sections further classes of finite groups (in particular almost simple groups) are investigated especially with ordinary character theory.
For the question whether the projection of $\mathrm{V}(\Z G)$ onto a faithful Wedderburn component (of $\Q G$ or $\C G$)  has a torsion free kernel, it becomes transparent that the use of  generic characters and generic character tables is extremely useful. This underlines the connection to other topics of computational representation theory of finite groups which are in the focus of recent research. It also shows that this is related with the construction of big torsion free normal subgroups of $\mathrm{V}(\Z G).$ This connects the investigation of torsion units of integral group rings with the 
other main topic in the area, the desription of the whole unit group of $\Z G$ in terms of generators and relations.  
This involves questions on the generation of units of infinite order, free non-abelian subgroups of $\mathrm{V}(\mathbb{Z}G)$, generators of subgroups which have finite index in $\mathrm{V}(\mathbb{Z}G)$ and others. For a recent detailed monograph on these latter topics see \cite{JespersDelRio1, JespersDelRio2}.

\section{Tools and Available algorithms}
An important tool to study the questions mentioned above are so-called partial augmentations. For an element $u = \sum\limits_{g \in G} u_g g \in \mathbb{Z}G$ and a conjugacy class $x^G$ in $G$ the integer
\[\varepsilon_x(u) = \sum_{g \in x^G} u_g \]
is called the partial augmentation of $u$ at $x$. Being class functions of $G$, partial augmentations are a natural object to study using representation theory. The connection between the questions mentioned in the introduction and partial augmentations is established by the following result. 

\begin{proposition}[Marciniak, Ritter, Sehgal, Weiss {\cite[Theorem 2.5]{MRSW}}]\label{MRSW_Prop}
A torsion unit $u \in \V(\ZZ G)$ is rationally conjugate to a group element if and only if $\varepsilon_x(u^d) \geq 0$ for all divisors $d$ of $o(u)$ and all $x \in G$. 
\end{proposition}
Note that for $u \in \mathrm{V}(\mathbb{Z}G)$ the condition $\varepsilon_x(u) \geq 0$ for all $x \in G$ is equivalent to the fact that one partial augmentation of $u$ is $1$ while all other partial augmentations are $0$ -- a situation which clearly applies for an element $g \in G$.

Thus it is of major interest to find restrictions on the partial augmentations of torsion units. For the orders of elements providing possibly non-vanishing partial augmentations the following is known.

\begin{proposition}\label{prop:pA} Let $u \in \mathrm{V}(\mathbb{Z}G)$ be a torsion unit of order $n$.
\begin{itemize}
\item [i)] $\varepsilon_1(u) = 0$, unless $u = 1$. \emph{(}Berman-Higman Theorem \cite[Proposition 1.5.1]{JespersDelRio1}\emph{)}. 
\item[ii)] If $\varepsilon_x(u) \neq 0$ for some $x \in G$ then the order of $x$ divides $n$ \cite[Theorem 2.3]{HertweckBrauer}.
\end{itemize}
\end{proposition}

From the properties of the $p$-power map in group algebras of characteristic $p$ one can obtain more restrictions on the partial augmentations given in terms of congruences modulo $p$.

\begin{lemma}[{cf. \cite[Proposition 3.1]{HeLPPaper} for a proof}]
Let $s$ be some element in $G$ and $u \in \V(\ZZ G)$ a unit of order $p^j \cdot m$ with $p$ a prime and $m \not= 1$. Then 
\[\smashoperator[r]{\sum\limits_{x^G,\ x^{p^j} \sim s}} \varepsilon_x(u) \equiv \varepsilon_{s}(u^{p^j}) \mod p.\]
\end{lemma}

In some special situations there are more theoretical restrictions on the partial augmentations of torsion units. We will only mention one of them which has not been used frequently yet, but turns out to be quite useful for our results.

\begin{proposition}[{\cite[Proposition 2]{Hertweck_OTU}}]\label{p-part} Suppose that $G$ has a normal $p$-subgroup $N$, and that $u$ is a torsion
unit in $\V(\ZZ G)$ whose image under the natural map $\ZZ G \to \ZZ G/N$ has strictly smaller
order than $u$. Then $\varepsilon_g(u) = 0$ for every element $g$ of $G$ whose $p$-part has strictly
smaller order than the $p$-part of $u$. 
\end{proposition}

\begin{remark}\label{rem:quotients}
An easy but often useful observation when working with quotient groups is the following. Let $N$ be a normal subgroup of $G$ and denote by $\varphi: \mathbb{Z}G \rightarrow \mathbb{Z}G/N$ the linear extension of the natural projection from $G$ to $G/N$. Then for an element $g \in G$ and a unit $u \in \mathrm{V}(\mathbb{Z}G)$ we have
\[\varepsilon_{\varphi(g)}(\varphi(u)) = \sum_{\substack{x^G, \\ \varphi(x) \sim \varphi(g)}} \varepsilon_x(u) \] 
where the sum runs over the conjugacy classes of $G$.
\end{remark}

\subsection{HeLP}
An idea to obtain more restrictions on the partial augmentations of torsion units in $\mathrm{V}(\mathbb{Z}G)$ using the values of ordinary characters of $G$ was introduced by I.S. Luthar and I.B.S. Passi \cite{LP89}. If $\chi$ denotes an ordinary character of $G$ and $D$ a representation of $G$ realizing $\chi$ then $D$ can be linearly extended to the group ring $\mathbb{Z}G$. This provides a ring homomorphism from $\mathbb{Z}G$ to a matrix ring and thus units of $\mathbb{Z}G$ are mapped to invertible matrices. Hence $D$ extends to a representation of $\mathrm{V}(\mathbb{Z}G)$ and $\chi$ extends to a character of $\mathrm{V}(\mathbb{Z}G)$. Let $x_1,...,x_h$ be representatives of the conjugacy classes of elements in $G$. Since $\chi$ is a $\ZZ$-linear function we obtain
\[\chi(u) = \sum_{i = 1}^{h} \varepsilon_{x_i}(u) \chi(x_i) \ \ \text{for} \ \ u \in \mathrm{V}(\mathbb{Z}G). \]
Denote by $\chi_1,...,\chi_h$ the irreducible complex characters of $G$ and by $\operatorname{X}(G)$ the character table of $G$. So from the arguments above we see 
\begin{align}\label{eq:HeLP} 
 \begin{pmatrix} \chi_1(u) \\ \chi_2(u) \\ \vdots \\ \chi_h(u) \end{pmatrix} = \begin{pmatrix} \chi_1(x_1) & \chi_1(x_2) & \hdots & \chi_1(x_h) \\ \chi_2(x_1) &  \chi_2(x_2) & \hdots & \chi_2(x_h) \\ \vdots & \vdots & \ddots & \vdots \\ \chi_h(x_1) &  \chi_h(x_2) & \hdots & \chi_h(x_h) \end{pmatrix}\begin{pmatrix} \varepsilon_{x_1}(u) \\ \varepsilon_{x_2}(u) \\ \vdots \\ \varepsilon_{x_h}(u) \end{pmatrix} = \operatorname{X}(G) \begin{pmatrix} \varepsilon_{x_1}(u) \\ \varepsilon_{x_2}(u) \\ \vdots \\ \varepsilon_{x_h}(u) \end{pmatrix}.
\end{align} 
Since the character table of a group is an invertible matrix, equation \eqref{eq:HeLP} provides restrictions on the partial augmentations of $u$ once we obtain restrictions on the character values $\chi_1(u),...,\chi_h(u)$.

For a unit $u$ of finite order these restrictions follow from the fact that $D(u)$ is a matrix of order dividing the order of $u$. Thus $D(u)$ is diagonalizable and its eigenvalues are $o(u)$-th roots of unity. So there are only finitely many possibilities for the values of $\chi(u)$. Hence going through these possibilities for all the irreducible complex characters of $G$ and applying equation \eqref{eq:HeLP}, one obtains finitely many possibilities for the partial augmentations of $u$. If we assume moreover that the partial augmentations of proper powers of $u$ are known, say by induction, then the restrictions on the eigenvalues of $D(u)$ can be significantly strengthened since they may be obtained as the pairwise product of the eigenvalues of $D(u^d)$ and $D(u^e)$ where $d, e$ denote integers not coprime with $o(u)$ such that $d + e \equiv 1 \mod o(u)$.

Clearly if one assumes that $K$ is an algebraically closed field of characteristic $p$ not dividing $o(u)$ then the arguments of the last paragraph still apply. Fixing a correspondence between the complex roots of unity of order not divisible by $p$ and the roots of unity in $K$, as it is custom in modular representation theory, one can view the character $\chi$ as a $p$-Brauer character having complex values. It was shown by Hertweck \cite[Section 3]{HertweckBrauer} that if one takes $x_1,...,x_h$ to be only the representatives of $p$-regular conjugacy classes in $G$ and $\chi_1,...,\chi_h$ to be the irreducible $p$-Brauer characters of $G$ then equation \eqref{eq:HeLP} also applies. This modular extension of the idea of Luthar and Passi is in particular useful for simple non-abelian groups. It does however not provide new restrictions for solvable groups by the Fong-Swan-Rukolaine Theorem \cite[Theorem 22.1]{CR1}. 

The HeLP method can be implemented into a computer program as it has been done in the \texttt{GAP}-package \texttt{HeLP} \cite{HeLPPaper}.  The HeLP method has been applied for single groups, e.g. in the study of the Zassenhaus Conjecture for small groups as in \cite{BovdiHoefertKimmerle}, \cite{HoeKi} or \cite{SmallGroups} or to study non-solvable groups as e.g. in \cite{KonovalovM22} or \cite{KimmerleKonovalov}. It might also be used to study infinite series of groups possessing generic character tables as it was done in \cite{HertweckBrauer}, \cite{4primaryHeLP} or \cite{FermatMersenne}. In this paper we will apply the HeLP method in Sections \ref{quasisimple} and \ref{BigNormal}.

Note that when one knows the partial augmentations of a torsion unit $u \in \mathrm{V}(\mathbb{Z}G)$ and all it powers, as e.g.\ after the application of the HeLP method, one may compute the eigenvalues, with multiplicities, of $D(u)$ for any ordinary representation $D$ of $G$. This observation is often useful when combining the HeLP method with other ideas described below.

\subsection{Other algorithmic methods: Quotients, Partially Central Units and the Lattice Method}
An inductive approach to questions about torsion units in $\mathrm{V}(\mathbb{Z}G)$ may be taken when one possesses information on the torsion units in $\mathrm{V}(\mathbb{Z}G/N)$ where $N$ is some normal subgroup in $G$, since a homomorphism from $G$ to $G/N$ naturally extends to a homomorphism from $\mathrm{V}(\mathbb{Z}G)$ to $\mathrm{V}(\mathbb{Z}G/N)$. If one can control the fusion of conjugacy classes in the projection from $G$ to $G/N$ then one can also obtain restrictions on the partial augmentations of elements in $\mathrm{V}(\mathbb{Z}G)$ assuming some knowledge about the partial augmentations of the units in $\mathrm{V}(\mathbb{Z}G/N)$, cf.\ Remark \ref{rem:quotients}. This approach was taken by many authors in particular when studying classes of groups closed under quotients as e.g.\ in \cite{Hertweck_Metacyclic}. In this paper also quotients play a significant role in all our results.

Assume that some torsion unit $u \in \mathrm{V}(\mathbb{Z}G)$ is central in some Wedderburn component $B$ of the complex group algebra $\mathbb{C}G$, but its spectrum in this component does not coincide with the spectrum of any element in $G$. An idea to disprove the existence of such units was first used manually by C. H{\"o}fert \cite{HoefertDiplom} and recently developed as a \texttt{GAP}-program by A. Herman and G. Singh \cite{HermanSingh}. This is sometimes called the Partially Central Method and uses an explicit representation of $G$ to show that no element in $\mathbb{C}G$ having only integral coefficients can realize the given central unit in $B$. Since $u$ has no other conjugates inside $B$, this also proves that $u$ can not be globally conjugate to an element in $\mathbb{Z}G$. This method turns out be be useful for the study of small groups as demonstrated in \cite{SmallGroups}, but is also not always successful.

A further algorithmic method, particularly useful for the study of the Prime Graph Question (cf. Problem 4 of Section 4), was introduced in \cite{Gitter} and is known as the Lattice Method. Let $p$ be a prime, $(K,R,k)$ be a $p$-modular system for $G$ and $u \in \mathrm{V}(\mathbb{Z}G)$ a torsion unit of order divisible by $p$. The idea of the Lattice method is roughly that when $B$ is a block of the modular group algebra $kG$, $D$ an ordinary irreducible representation of $G$ belonging to $B$ with corresponding $RG$-lattice $L$ and $S$ a simple $kG$-composition factor of $\bar{L}$, where $\bar{.}$ denotes the projection from $R$ onto $k$, then the spectrum of $D(u)$ provides restrictions on the isomorphism type of $L$ as $R\langle u \rangle$-lattice and thus on the isomorphism type of $S$ as $k\langle \bar{u} \rangle$-module. The use of other $RG$-lattices whose reduction to $kG$ involve $S$ as a composition factor may finally lead to a contradiction to the existence of $u$, since in some situations the restrictions obtained on the isomorphism type of $S$ as $k\langle \bar{u} \rangle$-module may contradict each other. The Lattice Method has successfully been applied in the study of the Prime Graph Question and the Zassenhaus Conjecture for non-solvable groups in \cite{Gitter, 4primaryGitter}. We will apply the Lattice Method in Section \ref{quasisimple}.

\section{Amitsur groups}\label{Amitsur groups}

Herstein pointed out that all finite subgroups of division rings in positive characteristic $p$ are cyclic $p'$-groups. In \cite{Ami} Amitsur described the finite subgroups of division algebras in characteristic $0$; these groups are nowadays often called \emph{Amitsur groups}. Recall that a group is a Z-group if all its Sylow subgroups are cyclic. We will use a weaker version of Amitsur's classification suitable for us.

\begin{theorem}[Amitsur, {\cite[Theorem~2.1.4]{SW}}] If a finite group $G$ is a subgroup of a division algebra of characteristic $0$ then
\begin{enumerate}
 \item[(Z)] $G$ is a Z-group
 \item[(NZ)] or G is isomorphic to one of the following groups
	\begin{enumerate}
	\item $\mathcal{O}^{*}= \langle s,t|(st)^2=s^3=t^4\rangle$ (binary octahedral group),
	\item $\SL(2,5)$,
	\item $\SL(2,3)\times M$, with $M$ a group in (Z) of order coprime to $6$ and $2$ has odd order modulo $|M|$,
	\item $C_m \rtimes Q$, where $m$ is odd, $Q$ a quaternion group of order $2^t$ such that an element of order $2^{t-1}$ centralizes $C_m$ and an element of order $4$ inverts $C_m$,
	\item $Q_8\times M$ with $M$ a group in (Z) of odd order and $2$ has odd order modulo $|M|$.
	\end{enumerate}
\end{enumerate}
\end{theorem}

In \cite[Theorem 3.5]{DJPM} it was proved that for an Amitsur group $G$, the order of a normalized torsion unit in $\ZZ G$ coincides with the order of an element in $G$. We now verify that even the first Zassenhaus Conjuecture holds for these groups.

\begin{theorem} Let $G$ be a finite subgroup of a division ring, then (ZC1) holds for $G$. \end{theorem}

\begin{proof} If $G$ is a Z-group then it is either cyclic or metacyclic and hence (ZC1) holds for $G$ \cite[Theorem 1.1]{Hertweck_Metacyclic}. The binary octahedral group was handled by Dokuchaev and Juriaans \cite[Proposition 4.2]{DJ} and $\SL(2,5)$ by Juriaans and Polcino-Milies \cite[Proposition 4.2]{JPM} (for those two groups even (ZC3) was verified). The groups in (NZ)(d) have cyclic normal subgroups of order $2^{t-1}m$ with an abelian quotient of order $2$, so we can again apply \cite[Theorem 1.1]{Hertweck_Metacyclic}. If $G = Q_8 \times M$ for some Z-group $M$ of odd order, then it is a direct product of a nilpotent group with a group for which the Zassenhaus Conjecture is known with coprime orders and the claim follows from \cite[Proposition 8.1]{Hertweck_Metacyclic}. So we are left with the groups in (NZ)(c).

Assume from now on that $G = \SL(2,3)\times M$, with $M$ a group in (Z) of order coprime to $6$ and let $u \in \V(\ZZ G)$ be a torsion unit. If $u$ has an order a divisor of $|M|$, then $(o(u), |\SL(2,3)|) = 1$.  Then we can consider the ring homomorphism $\ZZ G \to \ZZ G/\SL(2,3) \simeq \ZZ M$ induced by the projection $G \to G/\SL(2,3)$ to a group ring of a group for which the Zassenhaus Conjecture is known and apply \cite[Theorem 2.2]{DJ} to conclude that $u$ is rationally conjugate to an element of $G$. If $u$ has an order a divisor of $|\SL(2,3)|$ then a similar argument applies.

We now consider units having an order which has a common divisor with both, $6$ and the order of $M$. Note that $G$ has a normal Sylow $2$-subgroup $P$. Denote by $\varphi$ the natural ring homomorphism $\ZZ G \to \ZZ G/P$, which will also be denoted by bars, i.e.\ $\bar{x} = \varphi(x)$ for $x \in \ZZ G$.

Assume first that $u \in \V(\ZZ G)$ is of order $2m$ with $1 \not= m$ a divisor of $|M|$. If $\varepsilon_x(u) \not= 0$, then $o(x) \mid 2m$ by Proposition \ref{prop:pA}. The image $\bar{u}$ has order a divisor of $m$, in particular strictly smaller order than $u$. So by Proposition \ref{p-part} the only partial augmentations of $u$ that are potentially non-zero are those at classes of group elements of order $2m_0$ for $m_0$ a divisor of $m$. Let $w \in G$ be an element whose natural projection onto $\SL(2,3)$ is trivial. Then the conjugacy classes that are mapped under $\varphi$ onto the conjugacy class of $\bar{w}$ are exactly the classes of $w$, $zw$ and $tw$, where $z$ and $t$ denote elements of $G$ of order $2$ and $4$, respectively. Hence $\varepsilon_{\bar{w}}(\bar{u}) = \varepsilon_{w}(u) +  \varepsilon_{zw}(u) +  \varepsilon_{tw}(u) = \varepsilon_{zw}(u)$, since the order of $w$ and $tw$ is not of the form $2m_0$. As the Zassenhaus Conjecture holds for the metacyclic group $G/P$ we conclude that $u$ has exactly one non-vanishing partial augmentation and is rationally conjugate to a group element by Proposition \ref{MRSW_Prop}. 

If $u$ if of order $4m$ where $1 \not= m$ is a divisor of $|M|$ then analogues arguments as in the case $2m$ show that $u$ is rationally conjugate to an element of $G$.
Now assume that $u \in \V(\ZZ G)$ is of order $3m$ with $1 \not= m \mid |M|$. Then $\bar{u}$ is conjugate within $\QQ \bar{G}$ to an element of $\bar{G}$. As $(o(u), |P|) = 1$ we can use \cite[Theorem 2.2]{DJ} to conclude that $u$ is rationally conjugate to an element of $G$.

Finally assume that $o(u) = 6m$ with $1 \not= m$ a divisor of $|M|$. If $\varepsilon_x(u) \not= 0$ for some $x\in G$, then $o(x) \mid 6m$ by Proposition \ref{prop:pA}. The image $\bar{u}$ has strictly smaller order than $u$, so again by Proposition \ref{p-part} the only partial augmentations of $u$ that are potentially non-zero are those at classes of group elements of order $2m_0$ for $m_0$ a divisor of $3m$.  First consider an element $w \in G$ that projects to $1$ when mapped to $\SL(2,3)$. Then there are $3$ conjugacy classes of $G$ that are mapped on the conjugacy class of $\bar{w}$ in $\bar{G}$, namely those of $w$, $zw$ and $tw$, where $z$ and $t$ are elements of $G$ of order $2$ and $4$, respectively. Then $\varepsilon_{\bar{w}}(\bar{u}) = \varepsilon_{w}(u) +  \varepsilon_{zw}(u) +  \varepsilon_{tw}(u) = \varepsilon_{zw}(u)$. Now assume that $w \in G$ maps to an element of order $3$ in $\SL(2,3)$. Observe that for each element $s \in P$, $ws$ is either conjugate to $w$ or to $zw$. Hence exactly those two conjugacy classes map to the conjugacy class of $\bar{w}$. Thus $\varepsilon_{\bar{w}}(\bar{u}) = \varepsilon_{w}(u) +  \varepsilon_{zw}(u) = \varepsilon_{zw}(u)$. As the Zassenhaus Conjecture holds for $G/P$ we can conclude that $u$ has exactly one non-trivial partial augmentation. By Proposition \ref{MRSW_Prop}, $u$ is rationally conjugate to an element of $G$. The theorem is proved. \end{proof}

\section{From (IP) to (SIP) and (PQ) }

From the point of view of the unit group $\V (\Z G)$ the counterexample to the isomorphism problem (IP) simply shows that different group bases may be not isomorphic.  
Nevertheless a lot of positive results have been established and it is
justified to say that (IP) has almost a positive
answer. Indeed for each finite group $G$ there is an abelian extension
$E := A \sdp G$ such that (ZC2), and so also (IP), has a positive answer,
i.e. different group bases of $\Z E$ are conjugate within $\Q G.$ This
follows from the $F^*$-theorem which has been discovered by
K.W.Roggenkamp and L.L.Scott \cite{Sco:87}, \cite[Theorem 19]{Ro:91} and has now finally a published account
\cite[Theorem A and p.350]{HwJofAlg2016}, see also \cite[p.180]{KH}. With respect to semilocal
coefficient rings the $F^*$-theorem (in its automorphism version) 
may be stated as follows.  
\vskip1em

{\bf $F^*$-theorem.} \label{Fstar} 
   Let $G$ be a finite group. Denote by $\pi(G)$ the set of primes dividing the order of $G$. 
  Let $S $ be the semilocal ring  $ \Z_{\pi(G)} .$  
  Suppose that the generalized Fitting subgroup $F^*(G)$ is a $p$-group and let $\alpha $ be an $S$-algebra automorphism of $SG$ preserving augmentation.    Then $\alpha $ is given as the
  composition of an automorphism induced from a group automorphism of
  $G$ followed by a central automorphism (i.e.\ given by conjugation with a unit of $\Q G$).   
\vskip1em

The assumption on $G$ in the preceding theorem holds for all group
bases of $\Z G$ and for all group bases of $\Z (G \times G).$ Thus it
follows for groups whose generalized Fitting subgroup $F^*(G)$ is a
$p$-group that group bases of $\Z G$ are rationally conjugate, cf.\ \cite[5.3]{KiHab}.   
For a given group $G$ let $A$ be the additive group of $\mathbb{F}_pG$. Consider
the semidirect product $E = A \sdp G,$ where the action of $G$ is
just given by the multiplication of $G$ on $A .$ Clearly $C_E(A) = A$
and thus the $F^*$-theorem establishes (ZC2) and therefore a strong answer to (IP) for $\Z E.$         

The preceding paragraph shows that (IP) has a positive
solution for many important classes of finite groups. Thus the following subgroup variation came in the focus of research
within the last years.
\vskip1em   
{\bf Problem 1.} (Subgroup Isomorphism Problem \textbf{(SIP)}). Classify all finite
groups $H$ such that whenever $H$ occurs for a group $G$ as subgroup
of $\mathrm{V}(\Z G)$ then $H$ is isomorphic to a subgroup of $G .$ 

If $H$ has this property we say that (SIP) holds for $H .$    
\vskip1em
The known general results on (SIP).
\vskip1em
\begin{itemize}
\item[4.1.] (SIP) holds for cyclic groups of prime power order \cite{CohnLivingstone}. 
\item[4.2.] (SIP) is valid for $C_p \times C_p$, $p$ a prime \cite{KimmerleC2C2, HertweckCpcp}.  
\item[4.3.] (SIP) is valid for $C_4 \times C_2$ \cite{SIP}. 
\end{itemize}

This shows that with respect to general finite groups very
weak general facts are known about torsion units of the integral
group ring. On the related question when for a given specific group $G$
all torsion subgroups of $\mathrm{V}(\Z G)$ are isomorphic to a subgroup of $G$
much more is known. In Section 5 we settle this
question for all groups occurring as Frobenius complements.   
\vskip1em
Whether (SIP) holds for finite $p$-groups is certainly one of the
major open questions. Clearly this leads to Sylow like theorems for 
$\Z G.$ Even for conjugacy of finite $p$-groups within $\Q G$ no counterexample is known.    
\vskip1em

{\bf Problem 2.} \cite[p-ZC3,p.1170]{DJ} Is a Sylow like theorem \textbf{(SLT)} valid in $\mathrm{V}(\Z G)$, i.e.\ is each $p$-subgroup
of $\mathrm{V}(\Z G)$ rationally conjugate to a subgroup of $G$ ? 
\vskip1em

We say that (SLT) holds for a given group $G$ if in $\mathrm{V}(\Z G)$ Problem 2 has
an affirmative answer for all primes $p$ and that $(\text{SLT}_p)$ holds if
this is the case for a specific prime $p .$ If each subgroup of
prime power order $\mathrm{V}(\Z G)$ is isomorphic to a subgroup of $G$ we speak
of a weak Sylow like theorem \textbf{(WSLT)} and use the notion $(\text{WSLT}_p)$ if this holds for
a specific prime $p .$  
\vskip1em

Denote by $G_p$ a Sylow $p$-subgroup of $G$. Summary of known results on (SLT).
\vskip1em
\begin{itemize}
\item[4.4.] $(\text{SLT})_p$ holds when $G_p$ is normal \cite{Seh93}. 
\item[4.5.] (SLT) holds when $G$ is nilpotent-by-nilpotent \cite{DJ}.
\item[4.6.] $(\text{SLT})_p$ holds when $G_p$ is abelian and $G$ is $p$-constrained \cite[Proposition 3.2]{BaechleKimmerle}.
\item[4.7.] (SLT) holds for $\operatorname{PSL}(2,p^f)$ where $p$ denotes a prime \cite{SylowPSL} if $f = 1$ or $p =2$. It also holds for $\operatorname{PSL}(2,p^2)$ if $p \leq 5$ \cite{PassmanProc}. Moreover (WSLT) holds if $f = 2$ \cite{HertweckHoefertKimmerle}.
\item[4.8.] $(\text{SLT})_2$ is valid if $|G_2| \leq 8 $, unless $G \cong A_7$ \cite{SIP}. $(\text{WSLT})_p$ is valid if $G_p$ is cyclic \cite{KimmerleC2C2, HertweckCpcp} and $({\text{WSLT}})_2$ is proved if $G_2$ is generalized quaternion \cite[Theorem 4.1]{KimmerleSylow} or a dihedral group \cite{SIP}. 
\end{itemize}

For Frobenius groups we refer to the next section. 
The following two further special cases of (SIP) have been studied extensively
in the last decade.
\vskip1em
{\bf Problem 3.} (\textbf{(SIP-C)}, Problem 8 in \cite{Seh93}) Let $G$ be a finite
group. Is each cyclic subgroup
of $\mathrm{V}(\Z G)$ isomorphic to a subgroup of $G$?  
\vskip1em
{\bf Problem 4.} (Prime Graph Question \textbf{(PQ)}) Let $G$ be a finite group. Do $G$
and $\mathrm{V}(\Z G)$ have the same prime graphs? Equivalently, is (SIP)
valid for cyclic groups of order $p \cdot q ,$ where $p$ and $q$ are
different primes?  

\vskip1em
We say that (SIP-C)  or (PQ)  holds for a group $G$ if Problem 3 or Problem 4 respectively has a positive
answer for $G .$ Note that (ZC1) implies (SIP-C)  and this in turn implies (PQ) . 
So both problems
may be also considered as test problems for the first Zassenhaus
conjecture.   

\vskip1em
Summary of known results on (SIP-C)  and (PQ) . 
\vskip1em
\begin{itemize}
\item[4.9.] (SIP-C)  holds for soluble groups \cite{Hertweck_OTU}. Moreover (SIP-C)  is valid  
  for any soluble extension of a group $Q$ for which each torsion unit of order $n$ has non-vanishing partial augmentations on a class of elements of $Q$ of order $n$, cf. Lemma \ref{lem:LiftSIPC}. This is the case when (ZC1) holds for $\Z Q .$ \\
  (PQ) is valid for any soluble extension of a group $Q$ for which (PQ) holds.  \cite[Proposition 4.3]{KimmerlePQ}.
\item[4.10.] (SIP-C)  holds for Frobenius groups \cite[Corollary 2.5]{KimmerleKonovalov}.
\item[4.11.] (PQ)  holds for all simple groups $\PSL (2,p)$, for $p$ a prime \cite{HertweckBrauer}. \\ (PQ)  also holds for any almost simple group with socle isomorphic to $\PSL(2,p)$ or $\PSL(2,p^2)$ \cite{4primaryHeLP}.
\item[4.12.] If each almost simple image of the group $G$ has an order 
      divisible by three primes then (PQ)  has an affirmative 
      answer \cite[Theorem 3.1]{KimmerleKonovalov}, \cite{Gitter}.
\item[4.13.] (PQ)  holds for many almost simple groups whose socle has an order divisible by at most $4$ different primes \cite{4primaryGitter}. 
\end{itemize}
\vskip1em

We like to point out that the computational tools explained in Section 2 play a prominent role in proving these results. A typical example for this is 
4.11. By theoretical arguments the proof is reduced to almost simple groups whose order is divisible by exactly three primes, cf. \cite[\S 4]{KimKon:2015}. CFSG shows that there are only 8 such simple groups with this property. Now a computer algebraic examination of the almost simple groups arising from those simple groups yields successfully the 
result. The HeLP - method does not suffice to deal with all cases. In the case 
of automorphism groups of $A_6$ the final piece is obtained by the Lattice method. \\     
For further results on almost simple groups of small order see Section \ref{quasisimple}.

\section{Frobenius groups and complements}\label{sec:FrobeniusGroups}
The torsion units of the integral group rings of Frobenius groups were
considered in \cite{DJPM, JPM, KimmerlePQ, BovdiHertweck, PassmanProc}.
In particular SLT holds cf.\ \cite{PassmanProc} and (PQ) is known \cite{KimmerlePQ}. However, none of the Zassenhaus conjectures
has been established completely. For many specific Frobenius groups (ZC1) and (ZC3) are known.
The following Theorem \ref{th:Frobenius} should be seen as a first important step towards (ZC3) for Frobenius groups.   

\begin{theorem}\label{th:Frobenius}
Let $G$ be a Frobenius complement. Then each torsion subgroup of $\mathrm{V}(\Z G)$ is isomorphic to a subgroup of $G$. 
\end{theorem}

\begin{proof} By \cite[\S 18]{Passman} the structure of Frobenius
  complements $G$ is as follows. \\
 Denote by $W$ the Fitting subgroup of $G$. 

\begin{itemize}
\item[(1)] If $G_2$ is cyclic then $G$ is a Z-group.  
\item[(2a)] Suppose that $W_2$ is cyclic. Then $G$ is metabelian.
\item[(2b)] Suppose that $W_2 \cong Q_8$. Then    
$$ (i) \ \ G = \boxxedcomp{C_2}{\SL (2,3) \times M} 
  \ \ \mbox{or}  \ \ (ii) \ \ G = \boxedcomp{ \SL (2,3) \times M} 
  $$
$$ \mbox{or} \ \ (iii) \ \ G = \boxedcomp{Q_8  \times M},$$   
where $M$ is a  metacyclic Z-group of odd order coprime to the order of $\SL (2,3)$ and $Q_8 $ respectively.

\item[(2c)] Suppose that $W_2 \cong Q_{2^n}$ with $n \geq 4 .$
Then   
$$  \ \ G = \boxxedcomp{C_2}{C_{2^{n-1}} \times M} \ \ $$ 
where $M$ is a
  metacyclic Z-group of odd order and $G_2 \cong Q_{2^n} .$  
\item[(3)] If $G$ is insoluble then 
$$ (i) \ \ G = \boxxedcomp{C_2}{\SL (2,5) \times M} \ \ \mbox{or}  \ \ (ii) \ \ G = \boxedcomp{\SL (2,5) \times M} ,$$ 
where $M$ is a
  metacyclic Z-group of odd order coprime to the order of $\SL (2,5) .$  

\end{itemize}
\vskip1em

The following results settle several of these cases immediately. 
\begin{itemize}
\item[(5.2)] If $G$ is a direct product of two groups $H_1$ and $H_2$ of coprime order then each torsion subgroup of $\mathrm{V}(\Z G)$ is isomorphic to a subgroup of $G$ if and only if the same holds in $\mathrm{V}(\Z H_1) $ and $\mathrm{V}(\Z H_2) .$
\item[(5.3)] For Z-groups (ZC3) is valid \cite{Valenti}.
\item[(5.4)] If $G$ has an abelian normal subgroup $A$ such that $G/A$ is 
abelian then each torsion subgroup of $\mathrm{V}(\Z G)$ is isomorphic to a subgroup of $G .$ This follows from the small group ring sequence 
 $$ 0 \lra A \cong \Z G \cdot I(A) / I(G) \cdot I(A) \lra \Z G / I(G) \cdot I(A)  \lra  \Z G/A \lra 0 $$
together with the well known facts that torsion subgroups of $\mathrm{V}(\Z G) $ are trivial provided $G$ is an abelian  torsion group and that by \cite{MarSeh} \\ 
$\mathrm{V}(\Z G) \cap (1 + I(A)I(G))$ is torsionfree.       
\end{itemize}
So Case 1 follows from 5.3, Cases 2a and 2b(iii) from 5.4. Moreover (ZC3) is valid for $\SL(2,3)$, $\SL(2,5)$ \cite[Theorem 4.3]{DJPM} and for $p$-groups \cite{Weiss88}, \cite{ThomG}. Thus applying 5.2 and 5.3 we see that the the theorem is valid in the Cases 2b(ii), 2c(ii), 3(ii) rsp.
\vskip1em
Case 3(i). We first consider subgroups of $\mathrm{V}(\Z G)$ whose order divide 
$240 . $ Then factoring out the normal metacyclic group $M$ 
these subgroups are subgroups of $\mathrm{V}(\Z G/M)$. 
Then $G_1 = G/M$ is the double cover of the $S_5$ occurring as Frobenius complement.
(ZC1) holds for $\mathrm{V}(\Z G_1)$ \cite{BovdiHertweck}, (SIP-C)  holds for $G$ by \cite[Corollary 2.5]{KimmerleKonovalov} and a Sylow like theorem by \cite{PassmanProc}. Thus because the order of a torsion subgroup $H$ of $\mathrm{V}(\Z G_1)$ has to divide $|G_1| = 240 $ the remaining orders of subgroups of $\mathrm{V}(\Z G_1)$ are: 
$$ 240,120,80,60,48,40,30,24,20,15,12,10 \ \text{and} \ 6 .$$ 
If $|H| = 240$ then $H$ is a group basis. So we have to show that (IP) holds for $G_1 .$ This may be easily seen looking at the ordinary character table of $G_1.$ 
The normal subgroup correspondence shows that $\Z G_1 \cong \Z H$ implies that 
$H$ has to be as well a double cover of $S_5 .$ But the character tables of the two double covers of $S_5$ are different. Thus $H \cong G_1 .$
Each subgroup of even order of $\mathrm{V}(\Z G_1)$ contains the centre $Z$ of $G_1$ which is isomorphic to $C_2 .$ Thus subgroups of order 10 and 6 have to be cyclic. 
Factoring out $Z$ we see that a subgroup $H$ of order $120$ has to map onto a subgroup $\bar{H}$ of $\mathrm{V}(\Z S_5)$ of order $60.$ Because (ZC3) holds for $S_5$ by \cite{DJ} we see that $\bar{H} \cong A_5 .$ There are only two insoluble groups of order 120 which map onto $A_5 .$  The group $C_2 \times A_5$ has more than one involution. Thus it follows that $H \cong \SL (2,5) .$  \\
Also there are no subgroups of order $80, 60, 30, 15$ rsp.\ because $G_1/Z = S_5$  has no subgroups of order $40, 30, 15 $ rsp. \\
Assume now that $|H| = 40.$ Because a Sylow like theorem is valid in $\Z G_1$ we know that $H$ has a Sylow $2$-subgroup $H_2$ isomorphic to $Q_8 $ or to $C_8 .$ Suppose that $H_2 \cong Q_8 .$ Then $H_2/Z \cong C_2 \times C_2 .$ But $S_5$  has no subgroup of type $C_5 \times C_2 \times C_2 .$ Thus $H_2  \cong C_8 .$  Assume that $H$ is isomorphic to a dihedral group of order 40. Then $H/Z \cong D_{10} .$ But $H/Z$ has to be a Frobenius group of order 20 we conclude that $H \cong C_5 : C_8 $ which is indeed isomorphic a subgroup of $G_1 .$
Similarly one sees that subgroups of order 20 are isomorphic to $C_5 : C_4$, a subgroup of index 2 in $C_5 : C_8 .$ \\
Let $|H| = 48 .$ Then we know that $H/Z \cong S_4 .$  Moreover $Q_{16} \cong H_2 .$ Suppose that $H \cong Q_{16} \times C_3 .$ Then $H/Z \cong D_4 \times C_3 \not \cong S_4 .$ Clearly $H$ must have a normal subgroup of order $8 $ and contains a non-split central extension of $A_4$ of order $24 .$ This must be the binary tetrahedral group. Hence an examination of the groups of order 48 (e.g.\ with \texttt{GAP} \cite{GAP}) shows that $H$ is a binary octahedral group of order $48 .$ \\
Similarly one sees that subgroups of order $24$ are binary tetrahedral groups or $C_3 : Q_8 $ which maps onto a subgroup of $S_5$ isomorphic to $S_3 \times C_2 .$ Both occur as subgroups of $G_1 .$ \\
If $H$ has order 12, $H_2 \cong C_4 .$ Thus $H \cong C_{12} $ which must occur in $G_1$ because (SIP-C)  holds. \\

Now let $H$ be a subgroup of $\mathrm{V}(\Z G)$ whose order is not divisible by $2,3$ or $5.$ Then by reduction modulo $\SL(2,5)$ we get that $H$ is isomorphic to a subgroup of 
$\Z (G/\SL (2,5)).$ But $\bar{G} = G/\SL (2,5)$ is a $Z$ - group. Thus (ZC3) holds for $\bar{G} $ and $H$ is isomorphic to a subgroup of $M .$ \\
Finally, if $H$ is a subgroup such that $H$ maps onto a subgroup $\bar{H}$ of $\mathrm{V}(\Z \bar{G})$  (with $\bar{G} = G / \SL(2,5)$) of even order $m > 2.$ Then as before $\bar{H}$ is isomorphic to a subgroup of $\bar{G} .$ Let $K$ be the image of $H$ under the map onto $\Z G_1$ and $M_H$ the kernel of $H$ under this map. Clearly $M_H$ is isomorphic to the subgroup of index 2 of $\bar{H} .$  
 
$H$ is a semidirect product of the form $M_H \sdp K .$ The action of $K$ on $M_H$ is determined modulo $K/C_K(M_H)$ and therefore given by $\bar{H} .$ Thus its 
isomorphism type is given by $\bar{H}$ and $M_H .$    
\vskip1em 
Cases 2b(i) and 2c. That subgroups of 
$Q_{2^n} , n \geq 4$ are isomorphic to subgroups of $G$ follows in this case at once from \cite{MarSeh} (even (ZC3) holds by \cite{Weiss88}, \cite{ThomG}). 
Moreover (ZC3) holds for the binary octahedral group by \cite[Theorem 4.7]{DJ}. Now we 
can argue as in the case before and the proof is complete. 
\end{proof}

\section{(SIP-C) for almost quasisimple groups}\label{quasisimple}
In this section we will consider non-solvable groups and analyse how much the known methods can provide for our problems. We will concentrate on automorphic and central non-split extensions of non-abelian simple groups. Recall that a group $G$ is called \emph{almost simple} if there is a simple non-abelian group $S$ such that $G$ is isomorphic to a subgroup of the automorphism group of $S$ containing the inner automorphisms of $S$, i.e. $S \cong \operatorname{Inn}(S) \leq G \leq \operatorname{Aut}(S)$. Moreover a group is called qausisimple if it is a central non-split extension of a non-abelian simple group $S$. We will call a group \emph{almost quasisimple} if it is a central non-split extension of an almost simple group.

\begin{example}
One of the smallest almost quasisimple groups for which the Zassenhaus Conjecture is open is the symmetric group of degree $6$. The other group of the same size for which the Zassenhaus Conjecture is also open is the Mathieu group of degree $10$. In this example we will concentrate on the example of a possible involution in $\mathrm{V}(\mathbb{Z}S_6)$ which is of particular interest since its existence would also provide a counterexample to the Torsionfree Kernels Question, cf. Problem 5 in Section \ref{BigNormal},  and even more so since the involution would lie in the kernel of the the most natural representation of the group -- the permutation representation on $6$ points.

We provide details. Let $G$ be the symmetric group of degree $6$. Denote by $3a$ the conjugacy class of involutions in $G$ which have no fixed points in the natural action, i.e. elements of cycle type $(2,2,2)$, by $2b$ the class of involutions of cycle type $(2,2,1,1)$ and by $2c$ the conjugacy class of involutions of cycle type $(2,1,1,1,1)$. The HeLP method is not sufficient to exclude the existence of an involution $u \in \mathrm{V}(\mathbb{Z}G)$ satisfying $(\varepsilon_{2a},\varepsilon_{2b}, \varepsilon_{2c}) = (-1,1,1)$. Moreover the \texttt{GAP}-function \texttt{HeLP\_MultiplicitiesOfEigenvalues} provided by the \texttt{HeLP}-package allows to construct an element of $\mathbb{Q}G$ having the partial augmentations of $u$. Thus to show that $u$ does not exist in $\mathbb{Z}G$ it must be shown that the conjugacy class of this element in $\mathbb{Q}G$ has trivial intersection with the $\mathbb{Z}$-order $\mathbb{Z}G$ in $\mathbb{Q}G$. We give this element explicitly. For that let

\[\mathbb{Q}G \cong \mathbb{Q} \times \mathbb{Q} \times \mathbb{Q}^{5\times 5} \times \mathbb{Q}^{5\times 5} \times \mathbb{Q}^{5\times 5} \times \mathbb{Q}^{5\times 5} \times \mathbb{Q}^{9\times 9} \times \mathbb{Q}^{9\times 9} \times \mathbb{Q}^{10\times 10} \times \mathbb{Q}^{10\times 10} \times \mathbb{Q}^{16\times 16} \]

be the Wedderburn decomposition of $\mathbb{Q}G$. Here the first factor of the form $\mathbb{Q}^{5\times 5}$ is understood to correspond to the representation of $G$ obtained by cancelling out the trivial module from the $6$-dimensional natural permutation module of $G$. Moreover the fourth factor of the form $\mathbb{Q}^{5 \times 5}$ corresponds to the representation obtained from cancelling out the trivial module from the permutation module obtained by the other $6$-transitive action of $G$ (i.e. the one corresponding to the other conjugacy class of subgroups isomorphic to $S_5$ in $G$) and tensoring this module with the signum representation. Moreover the first factor of the form $\mathbb{Q}^{10 \times 10}$ is understood to correspond to an irreducible representation of $G$ which has character values $-2$ on the class $2a$. In this understanding a representative of the conjugacy class of $u$ in $\mathbb{Q}G$ is given by the following element:
\begin{align*}
(&(1), (1), \\ &\text{diag}(1,1,1,1,1), \text{diag}(1,\text{-}1,\text{-}1,\text{-}1,\text{-}1), \text{diag}(1,\text{-}1,\text{-}1,\text{-}1,\text{-}1),
\text{diag}(1,1,1,1,1), \\
&\text{diag}(1,1,1,1,1,\text{-}1,\text{-}1,\text{-}1,\text{-}1), \text{diag}(1,1,1,1,1,\text{-}1,\text{-}1,\text{-}1,\text{-}1), \\
&\text{diag}(1,1,1,1,1,1,\text{-}1,\text{-}1,\text{-}1,\text{-}1), \text{diag}(1,1,\text{-}1,\text{-}1,\text{-}1,\text{-}1,\text{-}1,\text{-}1,\text{-}1,\text{-}1), \\
&\text{diag}(1,1,1,1,1,1,1,1,\text{-}1,\text{-}1,\text{-}1,\text{-}1,\text{-}1,\text{-}1,\text{-}1,\text{-}1)).
\end{align*}

\end{example}

Regarding (SIP-C) more can be achieved for almost quasisimple groups.

\begin{theorem}\label{th:AlmostQuasiSimple}
Let $G$ be a almost quasisimple group and let $S$ be the only non-abelian composition of factor of $G$. If $S$ has smaller order than $\operatorname{PSL}(3,3)$ then (SIP-C) holds for $G$.
\end{theorem}

To prove Theorem \ref{th:AlmostQuasiSimple} we will use the HeLP and lattice methods. For some of the groups we study, the \texttt{GAP}-package implementing the HeLP method \cite{HeLPPaper} is not strong enough computationally and we will use the following result.

\begin{lemma}\label{lem:LiftSIPC}
Let $G$ and $H$ be groups such that $G$ contains a normal $p$-subgroup $N$ and $G/N \cong H$. Assume moreover that when $u \in \mathrm{V}(\mathbb{Z}H)$ is a torsion unit of order $k$ then there exists an element $h \in H$ of order $k$ such that $\varepsilon_h(u) \neq 0$. Then (SIP-C) holds for $G$.
\end{lemma} 
\begin{proof}
The proof follows the line of the proof of \cite[Theorem]{Hertweck_OTU}. Denote by $\varphi: \mathbb{Z}G \rightarrow \mathbb{Z}G/N \cong \mathbb{Z}H$ the linear extension of the natural homomorphism from $G$ to $G/N$. We will apply the bar-convention to this ring homomorphism. So let $u \in \mathrm{V}(\mathbb{Z}G)$ be a torsion unit. If $\bar{u}$ has the same order as $u$ then $u$ has the same order as an element in $G$ as by assumption $\bar{u}$ has the same order as an element in $H$. Assume on the other hand that $\bar{u}$ has strictly smaller order than the order of $u$. By assumption there is an element $h \in H$ such that $\varepsilon_h(\bar{u}) \neq 0$ and $h$ has the same order as $\bar{u}$. So there is a $g \in G$ with $\bar{g} = h$ and $\varepsilon_g(u) \neq 0$. Then the $p'$-part of the order of $g$ and the $p'$-part of the order of $u$ coincide. But by Proposition \ref{p-part} the $p$-parts of the order of $g$ and the order of $u$ also coincide and so $g$ has the same order as $u$.
\end{proof}

\begin{proof}[Proof of Theorem \ref{th:AlmostQuasiSimple}] There are ten non-abelian simple groups of order smaller than the order of $\operatorname{PSL}(3,3)$ which give rise to 50 almost quasisimple groups. All of these groups are listed in the ATLAS and all their character and Brauer tables are available in the \texttt{GAP} character table library \cite{CTblLib}. So in principle we can apply the HeLP package to all these groups. However it turns out that for central extension of the alternating and symmetric group of degree $7$ these computations do not finish in a day, while finishing in a few minutes for all other groups. Among the groups not having $A_7$ as a composition factor it turns out that HeLP is sufficient to prove (SIP-C) for all groups except groups containing $A_6$ as a normal subgroup of index $2$ and the groups $\operatorname{PSL}(2,16)$ and a group containing $\operatorname{PSL}(2,16)$ as a normal subgroup of degree $2$. We handle these cases separately:

The two groups containing $A_6$ as a normal subgroup of index $2$ for which the HeLP method is not sufficient to prove (SIP-C) are $\operatorname{PGL}(2,9)$ and the Mathieu group of degree $10$. For both these groups it remains to rule out the existence of units of order $6$ in their normalized unit group of the integral group ring. This has been already done in \cite{Gitter} using the lattice method. For the group $G = \operatorname{PSL}(2,16)$ it remains to rule out the existence of units if order $6$ in $\mathrm{V}(\mathbb{Z}G)$ and this has been also achieved using the lattice method in \cite[Theorem C]{4primaryGitter}.

Next let $G$ be a group of automorphisms of $\operatorname{PSL}(2,16)$ in which the group of inner automorphisms of $\operatorname{PSL}(2,16)$ has index $2$. To prove (SIP-C) for $G$ it remains to show that there are no torsion units of order $12$ in $\mathrm{V}(\mathbb{Z}G)$. HeLP provides us with two possibilities for the partial augmentations of a unit $u \in \mathrm{V}(\mathbb{Z}G)$ of order $12$, both of which have the same partial augmentations on $u^2$. So ruling out the latter partial augmentations for units of order $6$ will prove (SIP-C) for $G$. We will use the lattice method to do so. Denote by $2a$ the conjugacy class of involutions in $G$ which lies in $\operatorname{PSl}(2,16)$ and by $3a$ the conjugacy class of elements of order $3$ in $G$. The critical unit $u$ of order $6$ has partial augmentations equal to $0$ on all conjugacy classes except $2a$ and $3a$ and furthermore $(\varepsilon_{2a}(u), \varepsilon_{3a}(u)) = (4,-3)$ and the only class on which the partial augmentation of $u^3$ does not vanish is $2a$. For an ordinary character $\chi$ of $G$ denote by $\chi'$ it's $3$-modular reduction. Denote by $\mathbf{1}$ the trivial character of $G$. There are irreducible complex characters $\chi$ (a constituent of the lift of the Steinberg character of $\operatorname{PSL}(2,16)$ and $\psi$ of degree $16$ and $17$ respectively such that $\chi'$ is also irreducible as a $3$-modular Brauer character and $\psi' = \mathbf{1}' + \chi'$. Both characters $\chi$ and $\psi$ only take integral values. Thus by a theorem of Fong \cite[Corollary 10.13]{Isaacs} there exists a $3$-adically complete discrete valuation ring $R$ unramified over the $3$-adic integers such that there are $R$-representations $D_\chi$ and $D_\psi$ of $G$ realizing $\chi$ and $\psi$ respectively. The partial augmentations of $u$ and its powers allow us to compute the eigenvalues of $u$ under these representations, e.g.\ using the \texttt{GAP}-command \texttt{HeLP\_MultiplicitiesOfEigenvalues} from the \texttt{HeLP}-package. Denote by $\zeta$ a primitive $3$rd root of unity. Then
\begin{align*}
D_\chi(u) &\sim \mathrm{diag}(1,1,\zeta,\zeta,\zeta,\zeta^2,\zeta^2,\zeta^2, -1,-1,-1,-1,-\zeta,-\zeta,-\zeta^2,-\zeta^2), \\
D_\psi(u) &\sim \mathrm{diag}(1,1,1,1,1,\zeta,\zeta,\zeta^2,\zeta^2, -\zeta,-\zeta,-\zeta,-\zeta,-\zeta^2,-\zeta^2-,\zeta^2,-\zeta^2).
\end{align*}
Denote by $L_\chi$ and $L_\psi$ full $RG$-lattices corresponding to $D_\chi$ and $D_\psi$ respectively. When an $RG$-lattices $L$ is considered as an $R\langle u \rangle$-lattices it decomposes into a direct sum $L \cong L^+ \oplus L^-$ such that all direct summands of $L^*$ as $R\langle u^3 \rangle$-module are trivial while all direct summands of $L^-$ as $R\langle u^3 \rangle$-module are non-trivial by \cite[Proposition 1.3]{Gitter}. Denote by $\bar{.}$ the reduction modulo the maximal ideal of $R$, also with respect to modules, and let $k$ be the field obtained by factoring out the maximal module from $R$. Then from the eigenvalues given above and \cite[Proposition 1.4]{Gitter} we obtain that $\bar{L}_\chi^-$ has exactly two indecomposable direct summands of $k$-dimension at least $2$ while $\bar{L}_\psi^-$ has four such summands. But from $\psi' = \mathbf{1}' + \chi'$ we know that $\bar{L}_\chi^-$ and $\bar{L}_\psi^-$  must be isomorphic, since if $S$ denotes a simple $kG$-module corresponding to $\chi'$ then both these modules are isomorphic to $S^-$, i.e. the direct summand of $S$ as $k\langle \bar{u} \rangle$-module consisting of the non-trivial direct summands of $S$ as $k\langle \bar{u}^3 \rangle$-module. This provides a final contradiction to the existence of $u$.

It remains to show (SIP-C) for non-split central extensions of the alternating and symmetric group of degree 7. Denote by $2a$ the conjugacy class of double transpositions in $A_7$ and $S_7$, i.e. elements of cycle type $(2,2,1,1,1)$ and by $3a$ and $3b$ the conjugacy classes of elements of order $3$ where the latter is of cycle type $(3,3,1)$. Note that all these three classes are the same in $A_7$ and $S_7$. The Schur multiplier of both groups is cyclic of order $6$, i.e. the maximal cyclic non-split extension is by a cyclic group of order $6$. Thus by Lemma \ref{lem:LiftSIPC} it will be enough to show that when $u \in \mathrm{V}(\mathbb{Z}A_7)$ or $u \in \mathrm{V}(\mathbb{Z}S_7)$ is a unit of order $k$ then there exists a $g \in A_7$ or $g \in S_7$ respectively such that $\varepsilon_g(u) \neq 0$ and $g$ is of order $k$. Applying HeLP to $A_7$ and $S_7$ one finds that if $u$ is a unit not satisfying this condition then $u$ is of order $6$. Moreover $u$ satisfies 
\[(\varepsilon_{2a}(u), \varepsilon_{3a}(u), \varepsilon_{3b}(u)) \in \{(-2,2,1), (-2,1,2) \},\]
the partial augmentations of $u$ at all other elements vanish and $u^3$ is rationally conjugate to an element of $2a$ while $u^2$ is rationally conjugate to a $3$-element in the conjugacy class $C$ of $G$ such that $\varepsilon_C(u) = 1$. In particular we obtain that showing the non-existence of such a unit in $\mathrm{V}(\mathbb{Z}S_7)$ implies the non-existence of such a unit in $\mathrm{V}(\mathbb{Z}A_7)$.

So assume that $G = S_7$ and assume that $u$ is a unit of order $6$ in $\mathrm{V}(\mathbb{Z}G)$ as described in the last paragraph. Again we will use the lattice method to show that $u$ does not exist. The case when $\varepsilon_{3a} = 2$ will be called Case i) and the case that $\varepsilon_{3b}(u) = 2$ will be called ii). Denote by $\operatorname{sig}$ the character of $G$ corresponding to the signum representation. For an ordinary characters $\chi$ denote once more by $\chi'$ the corresponding $3$-Brauer character. $G$ possesses an irreducible $3$-Brauer character $\varphi$ of degree $13$ and two irreducible characters $\chi$ and $\psi$ of degree $14$ such that
\[\chi' = \mathbf{1}' + \varphi \ \ \text{and} \ \ \psi' = \operatorname{sig}' + \varphi. \]
Let again $D_\chi$ and $D_\psi$ be $R$-representations of $G$ corresponding to $\chi$ and $\psi$ respectively where $R$ is a $3$-adically complete discrete valuation ring unramified over the $3$-adic integers. From the given partial augmentations of $u$ and its powers we can compute the eigenvalues of $u$ under these representations. Denote by $\zeta$ a primitive $3$rd root of unity.\\
Case i):
\begin{align*}
D_\chi(u) &\sim \mathrm{diag}(1,1,\zeta,\zeta,\zeta,\zeta^2,\zeta^2,\zeta^2, -1,-1,-\zeta,-\zeta,-\zeta^2,-\zeta^2), \\
D_\psi(u) &\sim \mathrm{diag}(1,1,\zeta,\zeta,\zeta,\zeta^2,\zeta^2,\zeta^2, -1,-1,-1,-1,-\zeta,-\zeta^2).
\end{align*}
Case ii):
\begin{align*}
D_\chi(u) &\sim \mathrm{diag}(1,1,\zeta,\zeta,\zeta,\zeta^2,\zeta^2,\zeta^2, -1,-1,-1,-1,-\zeta,-\zeta^2), \\
D_\psi(u) &\sim \mathrm{diag}(1,1,\zeta,\zeta,\zeta,\zeta^2,\zeta^2,\zeta^2, -1,-1,-\zeta,-\zeta,-\zeta^2,-\zeta^2). 
\end{align*}
And moreover in both cases $\operatorname{sig}(u) = 1$. Let $L_\chi$ and $L_\psi$ be full $RG$-lattices corresponding to $\chi$ and $\psi$ respectively and denote by $\bar{.}$ the reduction modulo the maximal ideal of $R$. Let $k$ be the quotient of $R$ by its maximal ideal and let $S$ be a simple $kG$-module corresponding to $\varphi$. When viewed as $k\langle \bar{u} \rangle$-module $S$ decomposes into a direct sum $S \cong S^+ \oplus S^-$ such that $S^-$ contains all direct summands of $S$ as $k\langle \bar{u}^3 \rangle$-module which are not trivial. An analogues decomposition applies for $\bar{L}_\chi$ and $\bar{L}_\psi$. Then from \cite[Propositions 1.3]{Gitter}, the $3$-modular decomposition behaviour and the eigenvalues of $u$ under $\mathbf{1}$ and $\operatorname{sig}$ we conclude that $\bar{L}_\chi^- \cong S^- \cong \bar{L}_\psi^-$. However from \cite[Proposition 1.4]{Gitter} we know that in Case i) $\bar{L}_\chi^-$ has exactly two indecomposable summands of degree at least $2$ while $\bar{L}_\psi^-$ has only one such summand and in Case ii) $\bar{L}_\chi^-$ has exactly one indecomposable summands of degree at least $2$ while $\bar{L}_\psi^-$ has two such summand. This contradicts the existence of $u$ and finishes the proof. \end{proof}

\section{On big normal subgroups}\label{BigNormal}

The following question has not yet been systematically studied, but it appears naturally in the questions mentioned above and might be of independent interest.

\textbf{Problem 5} (Torsionfree Kernel Question \textbf{(TKQ)}): Let $K$ be a field
of characteristic zero and $G$ a finite group. Let $B$ be a
faithful block of $K G$ and $\pi$ be the projection from the units of $KG$ onto $B$. 
Is $\Ker \pi \cap \mathrm{V}(\Z G) $ torsion free? \\

We remark that a big area in the study of units in integral group
rings of finite groups is devoted to the study of ``big subgroups'' of
$\mathrm{V}(\mathbb{Z}G)$. This involves questions on the generation
of units of infinite order, free non-abelian subgroups of
$\mathrm{V}(\mathbb{Z}G)$, generators of subgroups which have finite
index in $\mathrm{V}(\mathbb{Z}G)$ and others. For more details 
we refer to recent monograph on these topics \cite{JespersDelRio1, JespersDelRio2}.

In this section we present a little idea how
torsion units may be used to find such big normal subgroups which are
torsion free and of finite index. We also make transparent how HeLP
and its companions may be used to answer TKQ. Note that a related question on the
existence of a torsion free complement has been studied extensively in
the 80's.   

Note that the hypothesis of the next lemma is valid if (ZC1) holds for
$\Z G.$ 

\begin{lemma} Let $G$ be a finite group. Suppose that elements of prime order of $\mathrm{V}(\Z G)$ are rationally
  conjugate to elements of $G .$ 
  
  Let $R$ be a field of characteristic zero. Suppose that $G$ is a subgroup of $GL(n,R)$ and let $\tau :\Q
    G \lra M_n(R) $ be the ring homomorphism which is the unique
    extension of  a given
    injective group homomorphism $G \lra GL(n,R) .$ 
    Let $K$ be the kernel of the group homomorphism $\tau\mid_{\mathrm{U}(\mathbb{Q}G)}: \mathrm{U}(\mathbb{Q}G) \rightarrow \operatorname{GL}(n,R)$. Then
    $K \cap \mathrm{V}(\mathbb{Z}G)$ is torsion free.       
 \end{lemma}

\begin{proof} By assumption $M_n(R)$ is a $\Q$-vector space. Thus
  $\tau $ extends uniquely. If $K$ is not torsion free then it has an
  element $u$ of prime order $p .$ By assumption 
there is a unit $v \in \Q G$ such that $v^{-1}uv = g \in G.$ But then
  $\tau (u) = 1 \neq \tau (g)$.
\end{proof}

\begin{example}\label{example_PSL_3_3} Let $G = \PSL(3,3)$. It is readily checked using the \texttt{GAP}-package \texttt{HeLP} that normalized units of $\V(\ZZ G)$ of order a prime $r$ are rationally conjugate to elements of $G$, except possibly for $r = 3$. However in this situation the command \texttt{HeLP\_MultiplicitiesOfEigenvalues} can be used to see that no torsion unit of order $3$ is contained in the kernel of any irreducible representation of $G$ of degree larger than $1$. Hence the previous corollary can be applied with any irreducible representation of $G$ different from the principal one and (TKQ) has a positive answer for each block of $\C G$, while (ZC1) is unknown for this group. 
\end{example}

\begin{proposition}\label{mink} Let $G$ be a finite group. Let $B \cong M_n(\Q)$ be a faithful block of the
  Wedderburn decomposition of $\Q G$ and let $p$ be a prime. 
Assume that either
\begin{itemize}
\item[i)] $p$ is an odd prime or $p = 4$ or
\item[ii)] $p = 2$ and $|G|$ is odd.
\end{itemize}  
Let $\pi $ be the projection of $\mathrm{U}(\Q G)$
  onto $B$. Then 
  \[(\Ker \pi \cap \mathrm{V}(\Z G))\cdot (1 + p \Z G \cap \mathrm{V}(\Z
  G))\]
   is a torsion free normal subgroup of $\mathrm{V}(\Z G)$ of finite index.    
\end{proposition}
 
\begin{proof} We may choose an integral representation of $G,$ 
  i.e.\ $\pi$ maps $\Z G$ into $\GL(n,\Z) .$ Consider the reduction $\kappa: \GL(n,\Z) \lra
  \GL(n,\Z / p\Z). $ By a classical result of Minkowski \cite[Lemma~9]{GurLor} the
  map $\kappa $ is injective on
  torsion elements if $p \neq 2.$ For $p=2$ the kernel is an elementary-abelian $2$-group. Thus under the 
  assumptions each
  torsion element of $\mathrm{V}(\Z G)$ injects into the finite group
  $\GL(n,\Z / p\Z). $    
\end{proof}

The preceding construction may be applied especially in the situation
of symmetric groups (with respect to almost each non-trivial block)
because $\Q$ is a splitting field for $S_n$ and only few blocks are
not faithful.    

\begin{proposition} 
	Let $G$ be a minimal simple group. Then $\V(\ZZ G)$ has a torsion-free normal subgroup of finite index constructed as in Proposition \ref{mink}.
\end{proposition}

\begin{proof} In \cite{Thom}, Thompson proved that a minimal simple group is isomorphic to $\PSL(2, q)$, $\Sz(q)$ or $\PSL(3,3)$. For the two series generic character tables are known, cf.\ e.g.\ \cite[XI, \S 5]{Huppert3}. Let $G$ be one of these groups and let $\chi$ be the Steinberg character. The character $\chi$ attains the same value $t$ on all elements of order a fixed prime $r$. Let $x_1, ..., x_s$ be representatives of the conjugacy classes of $G$ of elements of order $r$ and let $u \in \V(\ZZ G)$ be a torsion unit of order $r$. Then \[ \chi(u) = \sum_{j=1}^s \varepsilon_{x_j}(u)\chi(x_j) = t \not= \chi(1). \] Hence $u$ is not in the kernel of a representation $D$ affording $\chi$. $D$ can be realized over the rationals, hence by Proposition \ref{mink} we obtain a torsion-free normal subgroup of $\V(\ZZ G)$ of finite index.

In case $G = \PSL(3, 3)$ Example \ref{example_PSL_3_3} can be used to find a suitable block. 
\end{proof}

\bibliographystyle{plain}
\bibliography{bib_dfg}

\end{document}